\documentclass[12pt]{amsart}
\usepackage{times,fullpage}%,backref}
\usepackage{latexsym,amscd,amssymb}
\newtheorem{thm}{Theorem}

\newtheorem{lem}[thm]{Lemma}
\newtheorem{cor}[thm]{Corollary}

\theoremstyle{remark}

\theoremstyle{definition}

\newcommand{\R}{\mathbb{ R}}

\newcommand{\Z}{\mathbb{ Z}}

\DeclareMathOperator{\SL}{SL}

\DeclareMathOperator{\Sol}{Sol}
\DeclareMathOperator{\Nil}{Nil}
\newcommand{\Hyp}{\mathbb{ H}}

\title{Three-manifolds and K\"ahler groups}
\author{D.~Kotschick}
\address{Mathematisches Institut, {\smaller LMU} M\"unchen,
Theresienstr.~39, 80333~M\"unchen, Germany}
\email{dieter@member.ams.org}
\date{February 10, 2011; \copyright{\ D.~Kotschick 2010}}
\begin{document}

\begin{abstract}
We give a simple proof of a result originally due to Dimca and Suciu~\cite{DS}: a group that is both K\"ahler 
and the fundamental group of a closed three-manifold is finite. We also prove that a group that is both 
the fundamental group of a closed three-manifold and of a non-K\"ahler compact complex surface 
is $\Z$ or $\Z\oplus\Z_2$.
\end{abstract}

\maketitle

%\bigskip

\section{Introduction}

In the late 1980s the study of K\"ahler groups, that is, fundamental groups of closed K\"ahler manifolds, 
took off in spectacular fashion. While restrictions on such groups were previously known because of Hodge theory
and because of rational homotopy theory, several deep new results were proved around 1988. I will only
recall two of them here. These and many other results on K\"ahler groups are discussed in detail in~\cite{ABCKT}.

Firstly, generalising partial results of Johnson and Rees~\cite{JR}, Gromov proved:
\begin{thm}[Gromov~\cite{G}]\label{t:G}
A K\"ahler group does not split as a nontrivial free product.
\end{thm}
Secondly, building on work of Siu, Sampson and others, Carlson and Toledo proved:
\begin{thm}[Carlson--Toledo~\cite{CT}]\label{t:CT}
No fundamental group of a closed real hyperbolic $n$-manifold with $n\geq 3$ is a K\"ahler group.
\end{thm}
When these results were proved, several people, including Donaldson and Goldman, noticed the contrast between K\"ahler groups 
on the one hand and three-manifold groups on the other: the latter are closed under free products, and, according to Thurston, 
most three-manifolds with freely indecomposable fundamental group are hyperbolic. Moreover, a case by case check of the Thurston 
geometries as explained in~\cite{Scott} shows the following: closed three-manifolds carrying one of the geometries $S^2\times\R$, 
$\Hyp^2\times\R$, $\R^3$ or $\Sol^3$ have virtually odd first Betti number, and so their fundamental groups cannot be K\"ahler. 
Moreover, closed three-manifolds carrying one of the geometries $\Nil^3$ or $\SL_2(\R)$ have virtually positive first Betti numbers with trivial 
cup product from $H^1$ to $H^2$. Their fundamental groups cannot be K\"ahler by the Hard Lefschetz Theorem.
Now, the only Thurston geometry that has not been excluded is $S^3$, where every fundamental group is finite. Since all finite 
groups are K\"ahler, it was natural to expect that the intersection of three-manifold groups with the K\"ahler groups should consist
exactly of the finite groups appearing as fundamental groups of three-manifolds with geometry $S^3$. The obstacle to turning 
this expectation into a theorem, indeed a corollary of the above Theorems~\ref{t:G} and~\ref{t:CT}, came from three-manifolds with a 
non-trivial JSJ decomposition along incompressible tori. While one could imagine that those manifolds containing at least some 
hyperbolic piece might yield to a generalisation of the harmonic map techniques of Carlson and Toledo~\cite{CT}\footnote{A first 
step in this direction was soon taken by Hern\'andez-Lamoneda, although his paper~\cite{Her} was only published much later.}, 
the case of graph manifolds seemed intractable.

Twenty years ago one thought about such questions modulo Thurston's geometrisation conjecture. Since this has now 
been proved by Perelman~\cite{P1,P2,KL,MT}, an unconditional result can finally be obtained. Indeed, Dimca and Suciu recently proved:
\begin{thm} [Dimca--Suciu~\cite{DS}]\label{t:main1}
Assume that a group $\Gamma$ is the fundamental group both of a closed K\"ahler manifold and of a 
closed three-manifold. Then $\Gamma$ is finite, and, therefore, a finite subgroup of $O(4)$.
\end{thm}
Once one proves $\Gamma$ to be finite, it follows from Perelman's work~\cite{P1,P2,KL,MT} that $\Gamma$ is a finite subgroup of $O(4)$
acting freely on $S^3$. Note that by a classical construction due to Serre, every finite group is the fundamental group of a smooth complex projective 
variety, hence a closed K\"ahler manifold. By the Lefschetz hyperplane theorem one may assume this variety to be a surface.

To me, a surprising aspect of the proof given by Dimca and Suciu is that it does not follow the above outline at all, and makes little use 
of the Thurston approach to three-manifolds. In fact, their proof does not use Theorems~\ref{t:G} and~\ref{t:CT}. Instead, they 
consider separately the cases of trivial and of nontrivial first Betti number. If the first Betti number of the fundamental group
of a closed oriented three-manifold is positive, then they prove it is not K\"ahler using a lot of machinery of a very different sort: 
characteristic and resonance varieties, Catanese's 
approach to the Siu--Beauville theorem, a commutative algebra result of Buchsbaum--Eisenbud, \ldots . Then, for the case of 
zero first Betti number, Dimca and Suciu appeal to results of Reznikov and Fujiwara pertaining to Kazhdan's property $T$.
It is only at this point that their proof depends on geometrisation via Fujiwara's arguments.
%; in the case of positive first Betti number their argument is independent of the work of Thurston and Perelman.

The present paper arose from my attempt to understand the argument of Dimca and Suciu~\cite{DS}. From their 
treatment of the positive Betti number case I extracted the following strategy for obtaining a contradiction: 
{\it If $\Gamma$ has positive first Betti
number and is both the fundamental group of a closed oriented three-manifold  and of a closed K\"ahler manifold, then
$H^1(\Gamma;\R)$ comes from a complex curve. Therefore all cup products of classes in $H^1(\Gamma;\R)$ also 
come from a curve, and this is incompatible with three-dimensional Poincar\'e duality.}

One can actually implement this strategy in several different ways to prove Theorem~\ref{t:main1}. 
Here I will give quite a different implementation from that in~\cite{DS}, leading to a quick proof of the following:
\begin{thm}\label{t:alb}
If $\Gamma$ is a group with $b_1(\Gamma)>0$ whose real cohomology algebra $H^*(\Gamma;\R)$ satisfies 
$3$-dimensional oriented Poincar\'e duality, then $\Gamma$ is not a K\"ahler group.
\end{thm}
To put this into perspective, recall that many K\"ahler groups are Poincar\'e duality groups (of even dimension), cf.~\cite{JR,Tol,Kl}.
Also recall that, for every $k\geq 3$, Toledo~\cite{Tol} constructed examples of K\"ahler groups of cohomological dimension $2k-1$.
Moreover, his examples are duality (though not Poincar\'e duality) groups.

Of course, to exclude a group from being a K\"ahler group, it is enough that some finite index subgroup satisfy the 
assumptions of Theorem~\ref{t:alb}. Thus Theorem~\ref{t:alb} immediately gives:
\begin{cor}\label{c:asph}
Let $M$ be a closed aspherical three-manifold. If $M$ has a finite orientable covering that is not an $\R$-homology sphere, then
$\pi_1(M)$ is not a K\"ahler group.
\end{cor}
Theorem~\ref{t:alb} is more general than the Corollary because not every group whose real cohomology satisfies $3$-dimensional 
Poincar\'e duality is the fundamental group of an aspherical three-manifold. This issue is related to the three-dimensional Borel
conjecture; see Problem~3.77 on Kirby's problem list~\cite{Kirby}.

Corollary~\ref{c:asph} proves most of Theorem~\ref{t:main1}, since it handles not only manifolds with a nontrivial JSJ 
decomposition, but also gives a uniform treatment of geometric cases that no longer need to be checked case by case, so 
we obtain quite a simple proof of Theorem~\ref{t:main1}
for groups with virtually positive first Betti number. Using Perelman's geometrisation theorem, the case of first Betti number 
zero can actually be reduced to Theorem~\ref{t:CT}. In Section~\ref{s:proofs} below we first prove Theorem~\ref{t:alb}, and then spell out 
the resulting straightforward proof of Theorem~\ref{t:main1}, avoiding the difficult
arguments of Dimca--Suciu~\cite{DS}, and the appeals to the works of Reznikov and Fujiwara. Like the original proof of~\cite{DS},
the proof of Theorem~\ref{t:main1} given here uses geometrisation only to handle the case of trivial (virtual) first Betti number.

Using the Kodaira classification of non-K\"ahler complex surfaces we shall also prove the following:
\begin{thm} \label{t:main2}
Assume that a group $\Gamma$ is the fundamental group both of a closed complex surface $S$ and of a 
closed three-manifold. Then either $\Gamma$ is a finite subgroup of $O(4)$ and $S$ is a K\"ahler surface, 
or $\Gamma$ is $\Z$ or $\Z\oplus\Z_2$ and $S$ is a surface of class $VII$.
\end{thm}
This is interesting since in real dimension $6$ every finitely presentable group is the fundamental group of a 
compact complex manifold, as proved by Taubes~\cite{T}. Thus, for fundamental group questions, complex surfaces are at the 
watershed between curves and the unrestricted case of complex three-folds, just like three-manifolds are at the watershed between 
real surfaces and the case of four-manifolds, where all finitely presentable groups appear.

\section{Proofs}\label{s:proofs}

\begin{proof}[Proof of Theorem~\ref{t:alb}]
Suppose for a contradiction that $X$ is a closed K\"ahler manifold with fundamental group $\Gamma$, and let 
$\alpha_X\colon X\longrightarrow T^{b_1(\Gamma)}$
be its Albanese map. By the universal property of classifying maps, $\alpha_X$ factors up to homotopy into a composition
$$
X\stackrel{c_X}{\longrightarrow} B\Gamma \stackrel{a}{\longrightarrow} B\Z^{b_1(\Gamma)} = T^{b_1(\Gamma)} \ ,
$$
where $c_X$ is the classifying map of the universal covering of $X$.
%, and $a$ is induced by the Abelianisation followed by dividing out the torsion in first homology.
One concludes that $\alpha_X^* = c_X^*\circ a^*$ is trivial in real cohomology of degree $>3$ because $B\Gamma$ has no
such cohomology, and so the image of $\alpha_X$ cannot have complex dimension $2$ or more.
Thus the image of $\alpha_X$ is a complex curve $C$. 

It is well known, and easy to see, that a one-dimensional Albanese image must be smooth, and of course it has positive genus.
Thus the Albanese map $\alpha_X$ factors as 
$$
X\stackrel{c_X}{\longrightarrow} B\Gamma \stackrel{\hat a}{\longrightarrow} C \ .
$$
All the maps above induce isomorphisms in degree one cohomology. Moreover, $\alpha_X^*$ is nontrivial in degree $2$ 
cohomology, and so the same is true for $\hat a^*$. However, there is no class in $H^1(\Gamma;\R)$
that has a nontrivial cup product with the image of $\hat a^*$ in $H^2(\Gamma;\R)$, since this cup product comes from $C$, 
which has real dimension $=2$. This contradicts the assumption that $\Gamma$ satisfies $3$-dimensional Poincar\'e duality.
\end{proof}

\begin{proof}[Proof of Theorem~\ref{t:main1}]
We need to show that an infinite three-manifold group $\Gamma$ cannot be K\"ahler. Since finite coverings of 
K\"ahler manifolds are K\"ahler, we only need to exclude some finite index subgroup of $\Gamma$, 
and so three-manifolds can be replaced by their finite coverings. In particular
we may assume that all three-manifolds are orientable. 

We may restrict our attention to three-manifolds that are prime in the sense of being indecomposable under connected
sums, since a nontrivial free product is never a K\"ahler group by Theorem~\ref{t:G}. Such a prime
three-manifold is either $S^1\times S^2$, or is aspherical, cf.~\cite{M}. Since a K\"ahler group cannot be infinite cyclic,
we are reduced to the consideration of aspherical three-manifolds, so that, for all $3$-manifolds with positive (virtual) 
first Betti number, Theorem~\ref{t:main1} follows from Corollary~\ref{c:asph}, which in turn follows from Theorem~\ref{t:alb}
proved above.

To complete the proof of Theorem~\ref{t:main1} it remains to deal with groups with vanishing first Betti number.
Thus consider a closed oriented aspherical three-manifold $M$ with infinite fundamental group $\Gamma$ having $b_1(\Gamma)=0$.
If $M$ contains an incompressible torus, then by a result of Luecke~\cite{L}, see also~\cite{K}, $M$ has a finite covering with 
positive first Betti number, so that Corollary~\ref{c:asph} applied to this covering shows that
$\Gamma$ is not K\"ahler. Thus we are left with the case of an aspherical $M$ that contains no incompressible torus. 
Such manifolds are hyperbolic by the work of Perelman~\cite{P1,P2,KL}, and fundamental groups of hyperbolic three-manifolds are 
never K\"ahler by Theorem~\ref{t:CT}.
%This finishes the proof of Theorem~\ref{t:main1}.
\end{proof}

\begin{proof}[Proof of Theorem~\ref{t:main2}]
Suppose that $\Gamma$ is the fundamental group of both a compact complex surface $S$ and a closed three-manifold $M$.
As before we may assume $M$ to be orientable.

If $S$ is K\"ahler, then $\Gamma$ is finite by Theorem~\ref{t:main1}. Conversely, if $\Gamma$ is finite, then the first Betti number 
of $S$ vanishes, and so $S$ is K\"ahlerian, cf.~\cite{Buch}. 

If $S$ is not K\"ahlerian, then its first Betti number is odd, see again~\cite{Buch}. We now use the Enriques--Kodaira classification
to conclude that either $S$ is properly elliptic with $b_1(S)\geq 3$, or $S$ is of class $VII$ with $b_1(S)=1$, cf.~\cite{BPV,N}. In the first 
case $\Gamma$ is freely indecomposable and is a Poincar\'e duality group of dimension $4$ by results of Kodaira described 
in~\cite[Section~3 of Ch.~1]{ABCKT}. In the second case, it is known only that $\pi_1(S)$ cannot split into $\Gamma_1\star\Gamma_2$
with both $\Gamma_i$ containing proper subgroups of finite index; see~\cite[Thm.~1.35]{ABCKT}. However, since three-manifold groups
are residually finite~\cite{H}\footnote{The reference~\cite{H} treats only manifolds satisfying Thurston's geometrisation conjecture. 
By Perelman's work~\cite{P1,P2,KL} this is not a restriction.}, this is enough to conclude that in our case, where 
$\pi_1(S)=\Gamma=\pi_1(M)$, $\Gamma$ is indeed freely indecomposable.

Thus we may assume that $M$ is prime. If it is aspherical, then $\Gamma$ is a three-dimensional Poincar\'e duality group. 
This means that $\Gamma$ is not the fundamental group of a properly elliptic surface with $b_1(S)\geq 3$ since those groups are four-dimensional 
Poincar\'e duality groups. If $\Gamma$ is the fundamental group of a class $VII$ surface, then we have $b_1(\Gamma)=1$, and, by Poincar\'e 
duality on $M$, $b_2(\Gamma)=1$. Under the classifying map of the universal covering of $S$, $H^2(\Gamma;\R)$ injects into $H^2(S;\R)$, 
where it becomes an isotropic subspace for the cup product for dimension reasons. (Its cup square comes from the three-dimensional $M$.) 
Thus the intersection form of $S$ would have to be indefinite, which contradicts the known fact that the intersection forms of class $VII$ surfaces 
are negative definite; see~\cite[Lemma~1.45]{ABCKT}.

Thus we are left to consider the case of an $M$ that is prime but not aspherical. This means that $M$ is $S^1\times S^2$ 
if it is orientable; cf.~\cite{M}. However, for a nonorientable $M$ we could also have the nontrivial $S^2$-bundle over $S^1$, also with 
fundamental group $\Z$, and $S^1\times\R P^2$, with fundamental group $\Z\oplus\Z_2$; cf.~\cite{Scott}. Both $\Z$ and $\Z\oplus\Z_2$
occur as fundamental groups of Hopf surfaces. Conversely, every surface with one of these fundamental groups is of class $VII$; cf.~\cite{BPV,N}.
This completes the proof of Theorem~\ref{t:main2}.
\end{proof}

\section{Discussion}

\subsection{Avoiding the use of Theorem~\ref{t:G}}

In the proof of Theorem~\ref{t:main1} in Section~\ref{s:proofs}, I found it most straightforward to reduce to the consideration of prime 
three-manifolds by using Gromov's result on free products, stated as Theorem~\ref{t:G} in the introduction. However, one can completely 
bypass the use of Theorem~\ref{t:G}, as we now explain.

\begin{lem}\label{l:JR}
Assume that $\Gamma_1$ and $\Gamma_2$ each have a non-trivial finite quotient $f_i\colon\Gamma_i\longrightarrow Q_i$. Then
their free product $\Gamma_1\star\Gamma_2$ has a finite index subgroup with odd first Betti number.
\end{lem}
\begin{proof}
Consider the induced homomorphism $f\colon\Gamma_1\star\Gamma_2\longrightarrow Q_1\times Q_2$. By the Kurosh subgroup 
theorem, its kernel is of the form $F_k\star\Gamma$, where $F_k$ is a free group of rank $k=(\vert Q_1\vert -1)(\vert Q_2\vert -1)$,
and $\Gamma$ is a free product of copies of the kernels of the $f_i$. For a finite quotient $g\colon F_k\longrightarrow Q$ of order $d$ we 
consider the kernel $\Delta$ of $\bar g\colon F_k\star\Gamma\longrightarrow Q$, where $\bar g$ restricts to $F_k$ as $g$ and is 
trivial on $\Gamma$. Then $\Delta$ is isomorphic to $F_l\star\Gamma\star\ldots\star\Gamma$ with $d$ copies of $\Gamma$ appearing,
and $l=1+d(k-1)$. Thus $\Delta\subset \Gamma_1\star\Gamma_2$ is a finite index subgroup with 
$$
b_1(\Delta) = l+d\cdot b_1(\Gamma) = 1+d\cdot (k-1+b_1(\Gamma)) \ .
$$
Choosing $d$ to be even, we have found the desired subgroup.
\end{proof}

Since three-manifold groups are residually finite~\cite{H}, we have the following:
\begin{cor}
If $M$ is a non-prime three-manifold, then it has a finite covering with odd first Betti number.
\end{cor}
At the expense of appealing to residual finiteness, we can use this Corollary in place of Theorem~\ref{t:G} to exclude non-prime manifolds
from consideration in the proof of Theorem~\ref{t:main1}. More generally, without restricting to three-manifold groups, Lemma~\ref{l:JR}
tells us that an arbitrary free product whose free factors admit finite quotients cannot be a K\"ahler group. This is exactly the special case of 
Theorem~\ref{t:G} originally proved by Johnson and Rees~\cite{JR}. Indeed our proof of the Lemma is a simplification of the argument 
in~\cite{JR}.

\subsection{The necessity to discuss $\R$-homology spheres}

In the proof of Theorem~\ref{t:main1} it was necessary to consider separately the case of groups with zero first Betti number. This step would be
superfluous, if it were known that every closed three-manifold has a finite covering with positive first Betti number. If such a statement
were available, then one would not need Theorem~\ref{t:CT} for the proof of Theorem~\ref{t:main1} given here. 

Apparently the question of whether every closed three-manifold with infinite fundamental group has virtually positive first Betti 
number was raised long ago by Waldhausen, Thurston, and others; see Problems~3.2 and 3.50 in Kirby's problem list~\cite{Kirby}
and the references given there. Curiously, those references do not include~\cite{Hempel,L} and other papers quoted in~\cite{Hempel},
all of which contain a wealth of information about this problem. In any case, this problem seems to be still open.

\subsection{The second Betti number of infinite K\"ahler groups}

Carlson and Toledo have asked whether an infinite K\"ahler group has virtually positive second Betti number\footnote{The original reference
for their question is Section~18.16 in~\cite{Kol}, where only a more specific version is formulated.}. If this were known to 
be true, then, because of three-dimensional Poincar\'e duality, we would not have to consider $\R$-homology $3$-spheres in 
the proof of Theorem~\ref{t:main1}. Moreover, we would not need to use geometrisation, and we would not need Theorem~\ref{t:CT} either!
We refer to the paper of Klingler~\cite{Kl} for a recent discussion of this question of Carlson and Toledo.

Unfortunately, a slight misstatement occurs in~\cite[Prop.~3.44 (i)]{ABCKT}, which implicitly asserts a positive answer to the 
question of Carlson and Toledo. The statement $b_2(\pi_1(X))\geq 1$ there should be replaced by $b_2(X)\geq 1$ (which is trivial).
The Proposition in question was proved by Amor\'os~\cite{A}, whose paper does not contain the misstatement.

\bibliographystyle{amsplain}

\bigskip

\end{document}